\documentclass[10pt]{amsart}

\usepackage{amsmath,amsthm,amsfonts,amssymb,amscd}

\theoremstyle{plain}
\newtheorem{theorem}{Theorem}[section]
\newtheorem{lemma}[theorem]{Lemma}
\newtheorem{proposition}[theorem]{Proposition}

\theoremstyle{remark}

\newtheorem{remark}[theorem]{Remark}

\numberwithin{equation}{section}

\newcommand{\Z}{Z}
\newcommand{\z}{\mathbb {Z}}

\newcommand{\af}{\alpha}
\newcommand{\bt}{\beta}
\newcommand{\x}{{\cdot}}

\newcommand{\mlt}{\mathrm{Mlt}}
\newcommand{\inn}{\mathrm{Inn}}
\newcommand{\aut}{\mathrm{Aut}}

\begin{document}

\title[Free commutative automorphic loop]{The free commutative automorphic $2$-generated loop of nilpotency class $3$}

\author[Barros]{Dylene Agda Souza de Barros}
\address[Barros, Grishkov]{Institute of Mathematics and Statistics, University of Sao Paulo, Rua do Mat\~{a}o, 1010, Cidade Universit\'aria, S\~{a}o Paulo, SP, Brazil, CEP 05508-090}
\email[Barros]{dylene@ime.usp.br}

\author[Grishkov]{Alexander Grishkov}
\email[Grishkov]{shuragri@gmail.com}

\author[Vojt\v{e}chovsk\'y]{Petr Vojt\v{e}chovsk\'y}
\address[Vojt\v{e}chovsk\'y]{Department of Mathematics, University of Denver, 2360 S Gaylord St, Denver, Colorado 80208, USA}
\email[Vojt\v{e}chovsk\'y]{petr@math.du.edu}

\begin{abstract}
A loop is automorphic if all its inner mappings are automorphisms. We construct the free commutative automorphic $2$-generated loop of nilpotency class $3$. It has dimension $8$ over the integers.
\end{abstract}



\thanks{Dylene Agda Souza de Barros supported by FAPESP - process number 2010/16112-4. Alexander Grishkov supported by FAPESP and CNPq (Brazil). The research stay of Dylene Agda Souza de Barros and Alexander Grishkov at the University of Denver was partially supported by the Simons Foundation Collaboration Grant 210176 to Petr Vojt\v{e}chovsk\'y.}

\maketitle

\section{Introduction}

A \emph{loop} is a nonempty set $Q$ with a binary operation $\x$ such that for every $a\in Q$ the left and right translations $L_a$, $R_a:Q\to Q$, $bL_a = a\x b$, $bR_a=b\x a$ are bijections of $Q$, and there is an identity element $1\in Q$ satisfying $1\x a = a\x 1 = a$ for all $a\in Q$. We will also write the multiplication $\x$ as juxtaposition, and we will use $\x$ to indicate the priority of multiplications. For instance, $a(bc\x d)$ means $a\x ((b\x c)\x d)$.

The \emph{multiplication group} of $Q$ is the permutation group $\mlt(Q) = \langle L_a,\,R_a;\;a\in Q\rangle$ generated by all left and right translations. The stabilizer of $1$ in $\mlt(Q)$ is the \emph{inner mapping group} $\inn(Q)$. It is well known, cf. \cite{Br}, that $\inn(Q) = \langle L_{a,b},\,R_{a,b},\,T_a;\;a,\,b\in Q\rangle$, where $L_{a,b} = L_aL_bL_{ba}^{-1}$, $R_{a,b} = R_aR_bR_{ab}^{-1}$, $T_a = R_aL_a^{-1}$.

A loop $Q$ is \emph{automorphic} if $\inn(Q)\le\aut(Q)$, that is, if every inner mapping of $Q$ is an automorphism of $Q$. Note that groups are automorphic loops.

Automorphic loops were first studied in \cite{BrPa}, where it was proved, among other results, that automorphic loops form a variety and are \emph{power-associative}, that is, every element generates a group. It was shown in \cite{JoKiNaVo} that automorphic loops have the antiautomorphic inverse property $(ab)^{-1}=b^{-1}a^{-1}$. In particular, commutative automorphic loops have the \emph{automorphic inverse property}, or \emph{AIP}, $(ab)^{-1}=a^{-1}b^{-1}$. For an introduction to the structural theory and the history of automorphic loops, see \cite{KiKuPhVo}. For an introduction to the structural theory of commutative automorphic loops, see \cite{JeKiVo1}, \cite{JeKiVo2}.

This paper is concerned with free objects in the variety of commutative automorphic loops, in particular with the free commutative automorphic $2$-generated loop of nilpotency class $3$.

The \emph{center} of a loop $Q$ is the associative subloop $\Z(Q) = \{a\in Q;\;a\varphi = a$ for every $\varphi\in \inn(Q)\}$. Thus $Z(Q)$ consists of all elements $a\in Q$ that commute and associate with all other elements of $Q$. Define $\Z_0(Q) = 1$, $\Z_1(Q)=Z(Q)$, and for $i\ge 1$ let $\Z_{i+1}(Q)$ be the preimage of $\Z(Q/\Z_i(Q))$ under the canonical projection $Q\to Q/\Z_i(Q)$. Then a loop $Q$ is \emph{nilpotent of class $n$} if $\Z_{n-1}(Q)\ne Q = \Z_n(Q)$.

It was shown independently in \cite{Cs} and \cite{JeKiVo3} that for an odd prime $p$ every commutative automorphic loop of order $p^k$ is nilpotent. By \cite{JeKiVo3}, a commutative automorphic loop of order $p^2$ is a commutative group, but there exist nonassociative commutative automorphic loops of order $p^3$---these were constructed in \cite{JeKiVo2} and classified up to isomorphism in \cite{BaGrVo}.

One of the main tools used in the classification \cite{BaGrVo} was the description of the free commutative automorphic $2$-generated loop of nilpotency class $2$. This paper can therefore be seen as a natural continuation of the program begun in \cite{BaGrVo}. A related project is \cite{GrSh}, where heavy associator calculus was used to determine the bases and orders of free commutative Moufang loops with up to seven generators.

\medskip

For $n\ge 2$, let $F_n(x,y)$ be the free commutative automorphic loop of nilpotency class $n$ on free generators $x$, $y$.

For elements $a$, $b$, $c$ of a loop $Q$ denote by $(a,b,c)$ the \emph{associator} of $a$, $b$, $c$, that is, the unique element satisfying the equation $ab\x c = (a\x bc)(a,b,c)$.

We obtained the following description of $F_2(x,y)$ in \cite{BaGrVo}:

\begin{theorem}[{\cite[Theorem 2.3]{BaGrVo}}]
Let $F_2(x,y)$ be the free commutative automorphic loop of nilpotency class $2$ with free generators $x$, $y$, and let $u_1=(x,x,y)$, $u_2 = (x,y,y)$. Then every element of $F_2(x,y)$ can be written uniquely as $x^{a_1}y^{a_2}u_1^{a_3}u_2^{a_4}$ for some $a_1$, $a_2$, $a_3$, $a_4\in\mathbb Z$, and the multiplication in $F_2(x,y)$ is given by
\begin{displaymath}
    (x^{a_1}y^{a_2}u_1^{a_3}u_2^{a_4})(x^{b_1}y^{b_2}u_1^{b_3}u_2^{b_4})
    {=}x^{a_1+b_1}y^{a_2+b_2}u_1^{a_3+b_3-a_1b_1(a_2+b_2)}u_2^{a_4+b_4+a_2b_2(a_1+b_1)}.
\end{displaymath}
\end{theorem}

As we are going to see, to describe $F_3(x,y)$ is considerably more difficult.

Let us call an associator \emph{compounded} if it is of the form $(a,b,c)$ where at least one of $a$, $b$, $c$ is again an associator $(u,v,w)$. It is easy to see, cf. Proposition \ref{Pr:Compounded}, that a commutative loop is of nilpotency class at most $3$ if and only if all compounded associators are central.

Ultimately we prove in Theorem \ref{Th:Main} that every element of $F_3(x,y)$ is of the canonical form
\begin{displaymath}
    (x^{a_1}y^{a_2}\x u_1^{a_3}u_2^{a_4})v_1^{a_5}v_2^{a_6}v_3^{a_7}v_4^{a_8},
\end{displaymath}
where
\begin{align*}
    u_1 &= (x,x,y),     &u_2 &= (x,y,y),     &v_1 &= (x,x,u_1),\\
    v_2 &= (x,x,u_2),   &v_3 &= (y,y,u_1),   &v_4 &= (y,y,u_2),
\end{align*}
and where the multiplication formula is as in Lemma \ref{Lm:FullMult}. (The canonical form can be parsed unequivocally because the compounded associators $v_1$, $v_2$, $v_3$, $v_4$ are central.) This is accomplished in a series of steps:

In Section \ref{Sc:Symmetry} we study symmetries and linear properties of the associator map $(\_,\_,\_)$ in commutative automorphic loops of nilpotency class $3$. We conclude that in $F_3(x,y)$ it suffices to look at compounded associators of the form $(a,b,(c,d,e))$ where each $a$, $b$, $c$, $d$, $e$ is either $x$ or $y$. In Section \ref{Sc:Powers} we study powers within associators and derive a formula for $(a^i,b^j,c^k)$. In Section \ref{Sc:Reduction} we discover several nontrivial relations among compounded associators of $F_3(x,y)$, reducing all compounded associators to just $v_1$, $v_2$, $v_3$, $v_4$.

The multiplication formula for $F_3(x,y)$ is derived in Lemma \ref{Lm:FullMult}. A critical step in proving the main result, Theorem \ref{Th:Main}, consists of showing that the multiplication formula of Lemma \ref{Lm:FullMult} actually yields an automorphic loop. This follows by straightforward calculation (one merely needs to verify that the generators $L_{a,b}$ of the inner mapping group are automorphisms), but the calculation is extremely tedious and error-prone and we have therefore decided to delegate it to a computer. The Mathematica \cite{Mathematica} code that accomplishes the calculation can be downloaded from the website of the third-named author, \texttt{www.math.du.edu/\~{}petr}. Once we know that the formula of Lemma \ref{Lm:FullMult} yields an automorphic loop $Q$, it is easy to show that $F_3(x,y)$ is free and isomorphic to $Q$.

Recall that the \emph{associator subloop} $A(Q)$ of $Q$ is the least normal subloop of $Q$ containing all associators (so $Q/A(Q)$ is a group). The \emph{left nucleus $N_\lambda(Q)$}, \emph{middle nucleus $N_\mu(Q)$} and \emph{right nucleus $N_\rho(Q)$} consist of all elements $a\in Q$ such that $(a,b,c)=1$, $(b,a,c)=1$, $(b,c,a)=1$ for every $b$, $c\in Q$, respectively. Then the \emph{nucleus} $N(Q)$ is defined by $N(Q) = N_\lambda(Q) \cap N_\mu(Q)\cap N_\rho(Q)$. We conclude the paper by calculating the associator subloop, nuclei and the center of $Q=F_3(x,y)$.

\begin{remark}
In the beginning of this paper the proofs we offer give all the details, but later on we gradually rely more and more on the reader to provide intermediate steps in calculations. All such steps can be obtained in a straightforward fashion, albeit sometimes with considerable time commitment. More details will be found in the dissertation \cite{BaPhD} of the first-named author.
\end{remark}

\section{Symmetry and linearity in associators}\label{Sc:Symmetry}

Recall that the associator in any loop $Q$ is well-defined modulo $Z(Q)$, that is, $(a,b,c) = (az_1,bz_2,cz_3)$ for any $a$, $b$, $c\in Q$ and $z_1$, $z_2$, $z_3\in Z(Q)$. In any commutative loop the identity
\begin{equation}\label{Eq:Flex}
    (a,b,a)=1
\end{equation}
holds because $ab\x a = a\x ab = a\x ba$. It is well known that in any commutative loop of nilpotency class $2$ we have $(a,b,c) = (c,b,a)^{-1}$, $(a,b,c)(b,c,a)(c,a,b)=1$. We will use all these observations and the following well-known proposition without reference.

\begin{proposition}\label{Pr:Compounded}
Let $Q$ be a commutative loop.
\begin{enumerate}
\item[(i)] $Q$ has nilpotency class at most $2$ if and only if all associators are central.
\item[(ii)] $Q$ has nilpotency class at most $3$ if and only if all compounded associators $((a,b,c),d,e)$, $(a,(b,c,d),e)$, $(a,b,(c,d,e))$ are central.
\end{enumerate}
\end{proposition}
\begin{proof}
Suppose that $Q$ has nilpotency class at most $2$. Then $Q/Z(Q)$ is an abelian group. Since $A(Q)$ is the least normal subloop $S$ such that $Q/S$ is an abelian group, it follows that $A(Q)\le Z(Q)$. The converse is proved by reversing the argument.

Let us write $\overline{a}$ for $aZ(Q)\in Q/Z(Q)$. Suppose that $Q$ has nilpotency class at most $3$. Then $Q/Z(Q)$ has nilpotency class at most $2$ and thus $(\overline{a},\overline{b},\overline{c})\in Z(Q/Z(Q))$ for every $a$, $b$, $c\in Q$ by (i). This is equivalent to $((\overline{a},\overline{b},\overline{c}),\overline{d},\overline{e}) = (\overline{d},(\overline{a},\overline{b},\overline{c}),\overline{e}) = (\overline{d},\overline{e},(\overline{a},\overline{b},\overline{c})) = 1_{Q/Z(Q)}$ and thus to $((a,b,c),d,e)$, $(d,(a,b,c),e)$, $(d,e,(a,b,c))\in Z(Q)$. The converse is again proved by reversing the argument.
\end{proof}

We proceed to show that in a commutative automorphic loop of nilpotency class $2$ the associator is linear in all coordinates.

\begin{lemma}\label{Lm:Lba}
Let $Q$ be a loop and let $a$, $b$, $c\in Q$ be such that $(a,b,c)\in\Z(Q)$. Then $cL_{b,a} = c(a,b,c)^{-1}$, $aR_{b,c} = a(a,b,c)$, and $bL_aR_cL_a^{-1}R_c^{-1} = b(a,b,c)$.
\end{lemma}
\begin{proof}
Since $(a,b,c)$ is central, we have $ab\x c(a,b,c)^{-1} = (ab\x c)(a,b,c)^{-1} = a\x bc$, or $cL_bL_a = (c(a,b,c)^{-1})L_{ab}$, or $cL_{b,a} = c(a,b,c)^{-1}$. Also, $ab\x c = a(a,b,c)\x bc$, or $aR_bR_c = (a(a,b,c))R_{bc}$, or $aR_{b,c} = a(a,b,c)$. Finally, $ab\x c = a\x (b(a,b,c))c$ yields the last equality.
\end{proof}

\begin{proposition}\label{Pr:2}
Let $Q$ be an automorphic loop of nilpotency class $2$. Then $(ab,c,d) = (a,c,d)(b,c,d)$, $(a,bc,d) = (a,b,d)(a,c,d)$, $(a,b,cd) = (a,b,c)(a,b,d)$ for every $a$, $b$, $c$, $d\in Q$.
\end{proposition}
\begin{proof}
Since $Q$ is automorphic, the inner mapping $R_{c,d}$ is an automorphism. By Lemma \ref{Lm:Lba}, $ab(ab,c,d) = (ab)R_{c,d} = aR_{c,d}\x bR_{c,d} = a(a,c,d)\x b(b,c,d) = ab(a,c,d)(b,c,d)$ and $(ab,c,d) = (a,c,d)(b,c,d)$ follows. Consequently, $(a,b,cd) = (cd,b,a)^{-1} = ((c,b,a)(d,b,a))^{-1}  = (c,b,a)^{-1}(d,b,a)^{-1} = (a,b,c)(a,b,d)$. Finally, $1L_aR_dL_a^{-1}R_d^{-1}=1$ shows that $L_aR_dL_a^{-1}R_d^{-1}$ is also an inner mapping, so Lemma \ref{Lm:Lba} implies $bc(a,bc,d) = (bc)L_aR_dL_a^{-1}R_d^{-1} = bL_aR_dL_a^{-1}R_d^{-1}\x cL_aR_dL_a^{-1}R_d^{-1} = b(a,b,d)\x c(a,c,d) = bc(a,b,d)(a,c,d)$, and we are done upon canceling $bc$.
\end{proof}

Here is a local version of Proposition \ref{Pr:2}:

\begin{lemma}\label{Lm:Linear}
Let $Q$ be an automorphic loop of nilpotency class $3$, and let $a$, $b$, $c$, $d\in Q$.
\begin{enumerate}
\item[(i)] If $(a,c,d)$, $(b,c,d)\in Z(Q)$ then $(ab,c,d) = (a,c,d)(b,c,d)$.
\item[(ii)] If $(a,b,d)$, $(a,c,d)\in Z(Q)$ then $(a,bc,d) = (a,b,d)(a,c,d)$.
\item[(iii)] If $(a,b,c)$, $(a,b,d)\in Z(Q)$ then $(a,b,cd) = (a,b,c)(a,b,d)$.
\end{enumerate}
\end{lemma}
\begin{proof}
Let us prove (ii). Since $Q/Z(Q)$ is automorphic of nilpotency class $2$, Proposition \ref{Pr:2} implies $(a,bc,d) = (a,b,d)(a,c,d)z$ for some $z\in Z(Q)$. This means that $(a,bc,d)$ is central, too. Then the calculation at the end of the proof of Proposition \ref{Pr:2} is still valid (since all associators involved in the calculation are central). The proofs for (i) and (iii) are similar.
\end{proof}

\begin{lemma}\label{Lm:Symmetry}
Let $Q$ be a commutative automorphic loop of nilpotency class $3$. Then
\begin{align}
    ((a,b,c),d,e)^{-1}&=(e,d,(a,b,c)),\label{Eq:13A}\\
    (a,(b,c,d),e)&=(a,e,(b,c,d))((b,c,d),a,e).\label{Eq:MidA}
\end{align}
for every $a$, $b$, $c$, $d$, $e\in Q$.
\end{lemma}
\begin{proof}
Note that $z=((a,b,c),d,e)\in Z(Q)$ because $Q$ has nilpotency class $3$. The identity $(a,b,c)d\x e = ((a,b,c)\x de)z$ hence yields $((a,b,c)d\x e)z^{-1}=(a,b,c)\x de$, and we get $((a,b,c)d\x e)z^{-1}=(a,b,c)\x de=ed\x (a,b,c)=(e\x d(a,b,c))(e,d,(a,b,c))=((a,b,c)d\x e)(e,d,(a,b,c))$, which implies \eqref{Eq:13A}. We calculate
\begin{align*}
    &(ae\x(b,c,d))(a,e,(b,c,d))^{-1} = a\x e(b,c,d) = a\x (b,c,d)e \\
    &= (a(b,c,d)\x e)(a,(b,c,d),e)^{-1} = ((b,c,d)\x ae)((b,c,d),a,e)(a,(b,c,d),e)^{-1}\\
    &=(ae\x (b,c,d))((b,c,d),a,e)(a,(b,c,d),e)^{-1},
\end{align*}
which implies $(a,e,(b,c,d))^{-1}=((b,c,d),a,e)(a,(b,c,d),e)^{-1}$, or \eqref{Eq:MidA}.
\end{proof}

\begin{lemma}\label{Lm:2A}
Let $Q$ be a commutative automorphic loop of nilpotency class $3$. Then
\begin{equation}
    (a,(b,c,d),(e,f,g)) = ((a,b,c),d,(e,f,g)) = ((a,b,c),(d,e,f),g) = 1\label{Eq:2A}
\end{equation}
for every $a$, $b$, $c$, $d$, $e$, $f$, $g\in Q$.
\end{lemma}
\begin{proof}
Since $Q/\Z(Q)$ is of nilpotency class $2$, we have $(ef\x g)(e\x fg)^{-1} = (e,f,g)z$ for some $z\in Z(Q)$. Then
\begin{displaymath}
    (a,(b,c,d),(e,f,g)) = (a,(b,c,d),(ef\x g)(e\x fg)^{-1}).
\end{displaymath}
The automorphic inverse property and Lemma \ref{Lm:Linear} yield $(a,(b,c,d),(e,f,g))=1$. The identity $((a,b,c),(d,e,f),g)=1$ follows by \eqref{Eq:13A}. The argument for  $((a,b,c),d,(e,f,g))=1$ is similar.
\end{proof}

\begin{lemma}\label{Lm:SymA}
Let $Q$ be a commutative automorphic loop of nilpotency class $3$. Then
\begin{align}
    (a,b,c) &= (c,b,a)^{-1},\label{Eq:13}\\
    (a,b,c) &= (a,c,b)(b,a,c)\label{Eq:123}
\end{align}
for every $a$, $b$, $c\in Q$.
\end{lemma}
\begin{proof}
We have $ab\x c = (a\x bc)(a,b,c) = (cb\x a)(a,b,c) = (c\x ba)(c,b,a)\x (a,b,c) = (ab\x c)(c,b,a)\x (a,b,c)$, and the last term can be rewritten as $(ab\x c)\x (c,b,a)(a,b,c)$ by \eqref{Eq:2A}, so \eqref{Eq:13} follows. Similarly, $ab\x c = (a\x bc)(a,b,c) = (b\x ca)(b,c,a)\x (a,b,c) = ((c\x ab)(c,a,b)\x(b,c,a))(a,b,c)$, the last term equals $(ab\x c)((c,a,b)(b,c,a)\x (a,b,c))$ by \eqref{Eq:2A} and Lemma \ref{Lm:Linear}, so we have $(c,a,b)(b,c,a)\x (a,b,c)=1$. By Lemma \ref{Lm:2A}, the AIP and \eqref{Eq:13}, we get $(a,b,c) = ((c,a,b)(b,c,a))^{-1} = (c,a,b)^{-1}(b,c,a)^{-1} = (b,a,c)(a,c,b)$, which is \eqref{Eq:123}.
\end{proof}

Note that in any commutative loop of nilpotency class $3$ we have
\begin{equation}\label{Eq:LInn}
    aL_{b,c} = a(a,b,c)(bc,a,(a,b,c)),
\end{equation}
because
\begin{align*}
    aL_{b,c} &= (c\x ba)L_{cb}^{-1} = (ab\x c)L_{cb}^{-1} = ((a\x bc)(a,b,c))L_{cb}^{-1} = ((bc\x a)(a,b,c))L_{cb}^{-1}\\
     &= (bc\x a(a,b,c)(bc,a,(a,b,c)))L_{bc}^{-1} = a(a,b,c)(bc,a,(a,b,c)).
\end{align*}

\begin{proposition}\label{Pr:AProduct}
Let $Q$ be a commutative automorphic loop of nilpotency class $3$. Then
\begin{align*}
    (ab,c,d) =& (a,c,d)(b,c,d)((a,c,d),a,b)((b,c,d),b,a)\\
        &\x((a,c,d),b,c)((b,c,d),a,c)((a,c,d),b,d)((b,c,d),a,d),\\
    (a,b,cd) =& (a,b,c)(a,b,d)((a,b,c),c,d)((a,b,d),d,c)\\
        &\x((a,b,c),d,b)((a,b,d),c,b)((a,b,c),d,a)((a,b,d),c,a),\\
    (a,bc,d) =& (a,b,d)(a,c,d)((a,b,d),b,c)((a,c,d),c,b)\\
        &\x((a,b,d),c,a)((a,c,d),b,a)((a,b,d),c,d)((a,c,d),b,d).
\end{align*}
\end{proposition}
\begin{proof}
By \eqref{Eq:LInn}, the identity $(ab)L_{c,d} = aL_{c,d}\x bL_{c,d}$ can be rewritten as
\begin{align*}
    ab\x (ab,c,d)(cd,ab,(ab,c,d)) &= a(a,c,d)(cd,a,(a,c,d))\x b(b,c,d)(cd,b,(b,c,d))\\
    &= (a(a,c,d)\x b(b,c,d))(cd,a,(a,c,d))(cd,b,(b,c,d)).
\end{align*}
By Lemma \ref{Lm:Linear} and \eqref{Eq:2A}, \eqref{Eq:13}, we have
\begin{align*}
    a(a,c,d)\x b(b,c,d) &= (a(a,c,d)\x b)(b,c,d)(a(a,c,d),b,(b,c,d))^{-1}\\
    &= (a(a,c,d)\x b)(b,c,d)(a,b,(b,c,d))^{-1}((a,c,d),b,(b,c,d))^{-1}\\
    &= (a(a,c,d)\x b)(b,c,d)((b,c,d),b,a)\\
    &= (ba\x (a,c,d))(b,c,d)(b,a,(a,c,d))^{-1}((b,c,d),b,a) \\
    &= (ba\x (a,c,d))(b,c,d)((a,c,d),a,b)(b,c,d),b,a).
\end{align*}
Since $(ab,c,d) = (a,c,d)(b,c,d)z$ for some $z\in Z(Q)$, Lemma \ref{Lm:Linear} yields
\begin{align*}
    (cd,ab,(ab,c,d)) &= (cd,ab,(a,c,d)(b,c,d)) = (cd,ab,(a,c,d))(cd,ab,(b,c,d))\\
    &= (cd,a,(a,c,d))(cd,b,(a,c,d))(cd,a,(b,c,d))(cd,b,(b,c,d)).
\end{align*}
Upon substituting and canceling $ab$ and like associators, the identity $(ab)L_{c,d} = aL_{c,d}\x bL_{c,d}$ therefore becomes
\begin{displaymath}
    (ab,c,d)(cd,b,(a,c,d))(cd,a,(b,c,d)) = (a,c,d)(b,c,d)((a,c,d),a,b)((b,c,d),b,a).
\end{displaymath}
The formula for $(ab,c,d)$ now follows by Lemma \ref{Lm:Linear} and \eqref{Eq:13}.

Note that Lemma \ref{Lm:Linear} and \eqref{Eq:13} imply $((a,b,c),d,e)^{-1} = ((a,b,c)^{-1},d,e) = ((c,b,a),d,e)$. This observation and \eqref{Eq:13} applied to the formula for $(ab,c,d)$ yield the formula for $(a,b,cd)$.

Using \eqref{Eq:123}, we calculate
\begin{align*}
    &(a,bc,d) = (a,d,bc)(bc,a,d) = (a,d,b)(a,d,c)\x (b,a,d)(c,a,d)\\
        &\x ((a,d,b),b,c)((a,d,c),c,b)((a,d,b),c,d)((a,d,c),b,d)((a,d,b),c,a)((a,d,c),b,a)\\
        &\x ((b,a,d),b,c)((c,a,d),c,b)((b,a,d),c,a)((c,a,d),b,a)((b,a,d),c,d)((c,a,d),b,d).
\end{align*}
The first four associators associate by \eqref{Eq:2A}, so $(a,d,b)(a,d,c)\x (b,a,d)(c,a,d) = (a,d,b)(b,a,d)\x (a,d,c)(c,a,d) = (a,b,d)(a,c,d)$. We can similarly pair the compounded associators, using Lemma \ref{Lm:Linear}. For instance, $((a,d,b),b,c)((b,a,d),b,c) = ((a,d,b)(b,a,d),b,c) = ((a,b,d),b,c)$. The formula for $(a,bc,d)$ follows.
\end{proof}

We can now deal with products in all arguments of a compounded associator. For products of the form $(ab,c,(d,e,f))$ we can use Lemma \ref{Lm:Linear} (or Proposition \ref{Pr:AProduct}), and for products $(a,b,(cd,e,f))$ we note that $(cd,e,f) = (c,e,f)(d,e,f)z$ for some central element $z$ (the explicit form of $z$ follows from Proposition \ref{Pr:AProduct}) and calculate $(a,b,(c,e,f)(d,e,f)) = (a,b,(c,e,f))(a,b,(d,e,f))$ by Lemma \ref{Lm:Linear}. From now on, we will use these and similar identities, often without explicit reference.

\section{Powers within associators}\label{Sc:Powers}

Using Proposition \ref{Pr:AProduct}, we proceed to derive formulae for powers within associators. Define $\alpha:\mathbb Z\to\mathbb Z$, $\beta:\mathbb Z\to\mathbb Z$ by
\begin{equation}\label{Eq:AlphaBeta}
    \alpha(n) = (n^3-n)/3,\quad \beta(n)=n^2-n.
\end{equation}
Note that $\alpha(0)=\beta(0)=0$, $\alpha(n+1) = \alpha(n)+n^2+n$, $\beta(n+1) = \beta(n)+2n$, $-\alpha(-n) = \alpha(n)$ and $2n^2-\beta(-n) = \beta(n)$ for every $n\in\mathbb Z$.

\begin{lemma}\label{Lm:Powers}
Let $Q$ be a commutative automorphic loop of nilpotency class $3$. Then
\begin{align}
    (a^n,b,c) &= (a,b,c)^n((a,b,c),a,a)^{\alpha(n)}((a,b,c),a,b)^{\beta(n)}((a,b,c),a,c)^{\beta(n)},\label{Eq:PL}\\
    (a,b^n,c) &= (a,b,c)^n((a,b,c),b,b)^{\alpha(n)}((a,b,c),b,a)^{\beta(n)}((a,b,c),b,c)^{\beta(n)},\label{Eq:PM}\\
    (a,b,c^n) &= (a,b,c)^n((a,b,c),c,c)^{\alpha(n)}((a,b,c),c,a)^{\beta(n)}((a,b,c),c,b)^{\beta(n)}\label{Eq:PR}
\end{align}
for every $a$, $b$, $c\in Q$ and every $n\in\mathbb Z$.
\end{lemma}
\begin{proof}
We prove \eqref{Eq:PL}; the equations \eqref{Eq:PM}, \eqref{Eq:PR} are proven analogously. If $n=0$, \eqref{Eq:PL} holds. Suppose that \eqref{Eq:PL} holds for some $n\ge 0$. Note that $((a^i,b,c),a^j,d) = ((a,b,c)^i,a^j,d) = ((a,b,c),a,d)^{ij}$ for every $i$, $j\ge 0$, by Lemma \ref{Lm:Linear}, using our usual trick $(a^i,b,c) = (a,b,c)^iz$ for some $z\in Z(Q)$. By Proposition \ref{Pr:AProduct} we then have
\begin{align*}
    &(a^{n+1},b,c) =  (aa^n,b,c)\\
    &\phantom{x}= (a,b,c)(a^n,b,c)((a,b,c),a,a^n)((a^n,b,c),a^n,a)\\
    &\phantom{x}\phantom{=}\phantom{x}\x ((a,b,c),a^n,b)((a^n,b,c),a,b)((a,b,c),a^n,c)((a^n,b,c),a,c)\\
    &\phantom{x}= (a,b,c)(a^n,b,c)((a,b,c),a,a)^{n^2+n}((a,b,c),a,b)^{2n}((a,b,c),a,c)^{2n}\\
    &\phantom{x}= (a,b,c)^{n+1}((a,b,c),a,a)^{\alpha(n)+n^2+n}((a,b,c),a,b)^{\beta(n)+2n}((a,b,c),a,c)^{\beta(n)+2n}\\
    &\phantom{x}= (a,b,c)^{n+1}((a,b,c),a,a)^{\alpha(n+1)}((a,b,c),a,b)^{\beta(n+1)}((a,b,c),a,c)^{\beta(n+1)}.
\end{align*}
As for negative powers, first note that Proposition \ref{Pr:AProduct} gives
\begin{align*}
    1 =& (1,b,c) = (aa^{-1},b,c) = (a,b,c)(a^{-1},b,c)((a,b,c),a,a^{-1})((a^{-1},b,c),a^{-1},a)\\
    &\x ((a,b,c),a^{-1},b)((a^{-1},b,c),a,b)((a,b,c),a^{-1},c)((a^{-1},b,c),a,c)\\
    =&(a,b,c)(a^{-1},b,c)((a,b,c),a,b)^{-2}((a,b,c),a,c)^{-2}.
\end{align*}
Since associators associate with one another by \eqref{Eq:2A}, we deduce
\begin{displaymath}
    (a^{-1},b,c) = (a,b,c)^{-1}((a,b,c),a,b)^2((a,b,c),a,c)^2.
\end{displaymath}
Then for every $n>0$ we have
\begin{align*}
    &(a^{-n},b,c) = ((a^n)^{-1},b,c) = (a^n,b,c)^{-1}((a^n,b,c),a^n,b)^2((a^n,b,c),a^n,c)^2\\
    &\phantom{x}=(a,b,c)^{-n}((a,b,c),a,a)^{-\alpha(n)}((a,b,c),a,b)^{2n^2-\beta(n)}((a,b,c),a,c)^{2n^2-\beta(n)}\\
    &\phantom{x}=(a,b,c)^{-n}((a,b,c),a,a)^{\alpha(-n)}((a,b,c),a,b)^{\beta(-n)}((a,b,c),a,c)^{\beta(-n)},
\end{align*}
finishing the proof of \eqref{Eq:PL}.
\end{proof}

\begin{lemma}\label{Lm:AllPowers}
Let $Q$ be a commutative automorphic loop of nilpotency class $3$. Then
\begin{align*}
    (a^i,b^j,c^k) =& (a,b,c)^{ijk}((a,b,c),a,a)^{\alpha(i)jk}((a,b,c),a,b)^{\beta(i)j^2k}((a,b,c),a,c)^{\beta(i)jk^2}\\
    &\x ((a,b,c),b,a)^{i\beta(j)k}((a,b,c),b,b)^{i\alpha(j)k}((a,b,c),b,c)^{i\beta(j)k^2}\\
    &\x ((a,b,c),c,a)^{ij\beta(k)}((a,b,c),c,b)^{ij\beta(k)}((a,b,c),c,c)^{ij\alpha(k)}
\end{align*}
for every $a$, $b$, $c\in Q$ and $i$, $j$, $k\in\mathbb Z$.
\end{lemma}
\begin{proof}
By Lemmas \ref{Lm:Linear}, \ref{Lm:Powers} and Proposition \ref{Pr:AProduct}, $(a^i,b^j,c^k)$ is equal to
\begin{multline*}
    (a,b^j,c^k)^i((a,b^j,c^k),a,a)^{\alpha(i)}((a,b^j,c^k),a,b^j)^{\beta(i)}((a,b^j,c^k)a,c^k)^{\beta(i)}\\
    =(a,b^j,c^k)^i((a,b,c),a,a)^{\alpha(i)jk}((a,b,c),a,b))^{\beta(i)j^2k}((a,b,c),a,c)^{\beta(i)jk^2},
\end{multline*}
the term $(a,b^j,c^k)^i$ is equal to
\begin{multline*}
    [(a,b,c^k)^j((a,b,c^k),b,b)^{\alpha(j)}((a,b,c^k),b,a)^{\beta(j)}((a,b,c^k),b,c^k)^{\beta(j)}]^i\\
    = (a,b,c^k)^{ij}((a,b,c),b,b)^{i\alpha(j)k}((a,b,c),b,a)^{i\beta(j)k}((a,b,c),b,c)^{i\beta(j)k^2},
\end{multline*}
and the term $(a,b,c^k)^{ij}$ is equal to
\begin{multline*}
    [(a,b,c)^k ((a,b,c),c,c)^{\alpha(k)}((a,b,c),c,a)^{\beta(k)}((a,b,c),c,b)^{\beta(k)}]^{ij}\\
    = (a,b,c)^{ijk}((a,b,c),c,c)^{ij\alpha(k)}((a,b,c),c,a)^{ij\beta(k)}((a,b,c),c,b)^{ij\beta(k)}.
\end{multline*}
\end{proof}

\section{Reduction}\label{Sc:Reduction}

Since we ultimately want to describe the free loop $F_3(x,y)$, we will from now on start focusing on formulae that involve only two variables $x$, $y$. For fixed elements $x$, $y$ (not necessarily the generators of $F_3(x,y)$), let
\begin{align*}
    u_1 &= (x,x,y),\quad &u_2&=(x,y,y),\\
    z_1 &= (x,x,u_1),\quad &z_2&=(x,x,u_2),\quad &z_3&=(x,y,u_1),\quad &z_4&=(x,y,u_2),\\
    z_5 &= (y,x,u_1),\quad &z_6&=(y,x,u_2),\quad &z_7&=(y,y,u_1),\quad &z_8&=(y,y,u_2).
\end{align*}
No additional associators will be needed since $(x,y,x)=(y,x,y)=1$ by \eqref{Eq:Flex}, compounded associators of the form $((\_,\_,\_),\_,\_)$ and $(\_,(\_,\_,\_),\_)$ can be rewritten as products of compounded associators of the form $(\_,\_,(\_,\_,\_))$ by \eqref{Eq:13A} and \eqref{Eq:MidA}, compounded associators with associators in two components vanish by \eqref{Eq:2A}, and products within associators can be handled by Proposition \ref{Pr:AProduct}.

The reader might have noticed that in our product formulas (such as in Proposition \ref{Pr:AProduct}) we accumulate associators on the left, but we chose the canonical compounded associators with accumulated associators on the right. It is easy to convert between these two formats, since
\begin{displaymath}
    (a,b,(c,d,e)) = ((c,d,e),b,a)^{-1} = ((e,d,c)^{-1},b,a)^{-1} = ((e,d,c),b,a),
\end{displaymath}
by \eqref{Eq:13} and Lemma \ref{Lm:Linear}. Also note that
\begin{displaymath}
    (a,b,(c,d,e)) = (a,b,(e,d,c)^{-1}) = (a,b,(e,d,c))^{-1}
\end{displaymath}
thanks to \eqref{Eq:13} and Lemma \ref{Lm:Linear}.

\begin{lemma}\label{Lm:Reduction}
Let $Q$ be a commutative automorphic loop of nilpotency class $3$. Then for every $x$, $y\in Q$ we have $z_2=z_3=z_5$ and $z_4=z_6=z_7$.
\end{lemma}
\begin{proof}
Focusing first on the product in the third coordinate, we calculate, starting with \eqref{Eq:Flex}:
\begin{align*}
    1=&(xy,x,xy) = (xy,x,x)(xy,x,y)((xy,x,x),x,y)((xy,x,y),y,x)\\
    &\x ((xy,x,x),y,x)((xy,x,y),x,x))((xy,x,x),y,xy)((xy,x,y),x,xy)\\
    =&(xy,x,x)(xy,x,y)z_5z_3^{-1}z_3z_1^{-1}z_3z_7z_1^{-1}z_5^{-1} =(xy,x,x)(xy,x,y)z_1^{-2}z_3z_7\\
    =&(y,x,x)((y,x,x),y,x)((y,x,x),x,x)((y,x,x),x,x)\\
    &\x (x,x,y)((x,x,y),x,y)((x,x,y),y,x)((x,x,y),y,y)z_1^{-2}z_3z_7\\
    =&(y,x,x)(x,x,y)z_3z_1z_1z_5^{-1}z_3^{-1}z_7^{-1}z_1^{-2}z_3z_7 = 1\x z_3z_5^{-1}.
\end{align*}
Hence $z_3=z_5$, or $(x,y,(x,x,y)) = (y,x,(x,x,y))$. Interchanging $x$ and $y$ in this identity yields $(y,x,(y,y,x)) = (x,y,(y,y,x))$, which is equivalent to $z_6=z_4$.

In the following calculation we will also use \eqref{Eq:LInn} and Lemma \ref{Lm:Powers}. As $Q$ is automorphic, $yL_{x,x}\x (xy)L_{x,x} = (y\x xy)L_{x,x}$. On the left hand side of this identity we have
\begin{displaymath}
    yL_{x,x} = y(y,x,x)(x^2,y,(y,x,x)) = y(y,x,x)z_3^{-2}
\end{displaymath}
and
\begin{align*}
    (xy)L_{x,x}&= (xy)(xy,x,x)(x^2,xy,(xy,x,x)) = (xy)(xy,x,x)(x^2,xy,(y,x,x)) \\
        &= (xy)(y,x,x)((y,x,x),y,x)((y,x,x),x,x)^2z_1^{-2}z_3^{-2}\\
        &=(xy)(y,x,x)z_3z_1^2z_1^{-2}z_3^{-2} = (xy)(y,x,x)z_3^{-1},
\end{align*}
while on the right hand side we have
\begin{align*}
    (y\x xy)L_{x,x} &= (y\x xy)(y\x xy,x,x)(x^2,y\x xy,(y\x xy,x,x))\\
    &=(y\x xy)(y^2x\x (y,y,x)^{-1},x,x)(x^2,y^2x,(y^2,x,x))\\
    &=(y\x xy)\x (y^2x,x,x)((x,y,y),x,x)z_3^{-8}z_1^{-4}\\
    &=(xy^2\x (x,y,y))\x (xy^2,x,x)z_2^{-1}z_3^{-8}z_1^{-4}\\
    &=(xy^2\x (x,y,y))\x (y^2,x,x)((y^2,x,x),y^2,x)((y^2,x,x),x,x)^2z_2^{-1}z_3^{-8}z_1^{-4}\\
    &=(xy^2\x (x,y,y))(y,x,x)^2((y,x,x),y,y)^2((y,x,x),y,x)^4z_3^4z_1^4z_2^{-1}z_3^{-8}z_1^{-4}\\
    &=xy^2\x (x,y,y)(y,x,x)^2z_7^2z_3^4z_3^{-4}z_2^{-1} = xy^2\x (x,y,y)(y,x,x)^2 z_7^2z_2^{-1}.
\end{align*}
Returning to the left hand side, we rewrite it as
\begin{align*}
    &y(y,x,x)\x (xy)(y,x,x)z_3^{-3} = (y(y,x,x)\x xy)(y,x,x)(y(y,x,x),xy,(y,x,x))^{-1}z_3^{-3}\\
    &\phantom{x}=(xy\x y(y,x,x))(y,x,x)(y,xy,(y,x,x))^{-1}z_3^{-3}\\
    &\phantom{x}=(xy\x y)(y,x,x)(xy,y,(y,x,x))^{-1}\x (y,x,x)z_5z_7z_3^{-3}\\
    &\phantom{x}=(xy\x y)(y,x,x)^2z_3z_7z_5z_7z_3^{-3}\\
    &\phantom{x}=(xy^2)(x,y,y)\x (y,x,x)^2z_3^{-2}z_5z_7^2 = xy^2\x (x,y,y)(y,x,x)^2z_3^{-2}z_5z_7^2.
\end{align*}
Comparing the two sides now yields $z_2^{-1}=z_3^{-2}z_5$. But $z_3^{-2}z_5 = z_3^{-1}$ by the first part of this lemma, and hence $z_2=z_3$. Switching $x$ and $y$ in the identity $z_2=z_3$ gives $z_7=z_6$.
\end{proof}

With the reduction of Lemma \ref{Lm:Reduction} in mind, we set for any fixed $x$, $y$
\begin{align*}
    u_1 &= (x,x,y),\quad &u_2 &= (x,y,y),\quad &v_1 &= (x,x,u_1),\\
    v_2 &= (x,x,u_2),\quad &v_3 &= (y,y,u_1),\quad &v_4&=(y,y,u_2).
\end{align*}
We are now ready to describe canonical elements of the free loop $F_3(x,y)$.

\begin{lemma}\label{Lm:Canonical}
Every element of $F_3(x,y)$ can be written in the canonical form
\begin{displaymath}
    (x^{a_1}y^{a_2}\x u_1^{a_3}u_2^{a_4}) v_1^{a_5}v_2^{a_6}v_3^{a_7}v_4^{a_8},
\end{displaymath}
where $a_i\in\mathbb Z$.
\end{lemma}
\begin{proof}
Let $X=\{x,x^{-1},y,y^{-1}\}$. We first note that any associator can be written as $u_1^{b_1}u_2^{b_2}\prod v_i^{c_i}$.
Indeed, since $F_3(x,y)$ has nilpotency class three, no compounded associators appear within associators. Using Lemma \ref{Lm:Linear}, Proposition \ref{Pr:AProduct} and their consequences, every associator can be written as a product of compounded associators and ordinary associators with all variables in $X$. In fact, equations \eqref{Eq:2A}, \eqref{Eq:Flex}, \eqref{Eq:13A} and \eqref{Eq:MidA} imply that every associator is a product of $u_1$, $u_2$, the $z_i$s and their inverses. Since associators associate among themselves by \eqref{Eq:2A}, this product is of the form $u_1^{b_1}u_2^{b_2}\prod z_i^{d_i}$ for suitable exponents in $\mathbb Z$, and hence of the form $u_1^{b_1}u_2^{b_2}\prod v_i^{c_i}$ by Lemma \ref{Lm:Reduction}.

To establish the lemma, it suffices to show that a product of two canonical words is also canonical. First, $[x^{a_1}y^{a_2}\x u_1^{a_3}u_2^{a_4}\prod v_i^{c_i}]\x[x^{b_1}y^{b_2}\x u_1^{b_3}u_2^{b_4}\prod v_i^{d_i}] = (x^{a_1}y^{a_2}\x u_1^{a_3}u_2^{a_4})(x^{b_1}y^{b_2}\x u_1^{b_3}u_2^{b_4})\prod v_i^{c_i+d_i}$, so it suffices to show that the product $(x^{a_1}y^{a_2}\x u_1^{a_3}u_2^{a_4})(x^{b_1}y^{b_2}\x u_1^{b_3}u_2^{b_4})$ has the desired form. We can rewrite this word as $((x^{a_1}y^{a_2}\x x^{b_1}y^{b_2})\x u_1^{a_3}u_2^{a_4})\x u_1^{b_3}u_2^{b_4}w$ with some product $w$ of compounded associators, and further to $(x^{a_1}y^{a_2}\x x^{b_1}y^{b_2})\x u_1^{a_3+b_3}u_2^{a_4+b_4}w$, using Lemma \ref{Lm:Linear} and \eqref{Eq:2A}. Now, $x^{a_1}y^{a_2}\x x^{b_1}y^{b_2}$ can be written as $(\cdots((x^{a_1+b_1}y^{a_2+b_2})t_1)t_2\cdots)t_k$, where each $t_i$ is an associator. Using Lemma \ref{Lm:Linear} and \eqref{Eq:2A} again, we further rewrite this as $(x^{a_1+b_1}y^{a_2+b_2})(t_1t_2\cdots t_k)$. The rest is easy.
\end{proof}

\section{The Main result}\label{Sc:Main}

The calculation described in the proof of Lemma \ref{Lm:Reduction} is straightforward but rather tedious. Before we attempt it, we note:

\begin{lemma}\label{Lm:MidA}
In the loop $F_3(x,y)$ all associators are in the middle nucleus.
\end{lemma}
\begin{proof}
Thanks to \eqref{Eq:Flex}, \eqref{Eq:13} and Lemmas \ref{Lm:Linear}, \ref{Lm:Canonical} and Proposition \ref{Pr:AProduct}, it suffices to show that $(x,u_1,y) = (x,u_2,y)=1$. By \eqref{Eq:MidA} and \eqref{Eq:13}, $(x,u_1,y) = (x,y,u_1)(u_1,x,y) = (x,y,u_1)(y,x,u_1)^{-1} = z_3z_5^{-1} = 1$, and also $(x,u_2,y) = (x,y,u_2)(u_2,x,y) = (x,y,u_2)(y,x,u_2)^{-1} = z_4z_6^{-1}=1$, where we used Lemma \ref{Lm:Reduction}.
\end{proof}

Recall the mappings $\alpha$, $\beta$ of \eqref{Eq:AlphaBeta}.

\begin{lemma}\label{Lm:Mult}
In $F_3(x,y)$ we have for every $a_1$, $a_2$, $b_1$, $b_2\in \mathbb Z$
\begin{align*}
    &x^{a_1}y^{a_2}\x x^{b_1}y^{b_2}=x^{a_1+b_1}y^{a_2+b_2}\x u_1^{-a_1b_1(a_2+b_2)} u_2^{a_2b_2(a_1+b_1)}\\
    &\phantom{x}\x v_1^{(a_2+b_2)(b_1\af(a_1)+a_1\af(b_1))+a_2(a_1\bt(b_1)+b_1^2\bt(a_1))+b_2(b_1\bt(a_1)+a_1^2\bt(b_1))}\\
    &\phantom{x}\x v_2^{2a_1a_2b_1b_2(a_1+b_1)+(a_2+b_2)(a_1\bt(b_1)+b_1\bt(a_1))+(\bt(a_2)+\bt(b_2))(a_1b_1^2+b_1a_1^2)-a_2b_2\af(a_1+b_1)}\\
    &\phantom{x}\x v_3^{-2a_1a_2b_1b_2(a_2+b_2)-(a_1+b_1)(a_2\bt(b_2)+b_2\bt(a_2))-(\bt(a_1)+\bt(b_1))(a_2b_2^2+b_2a_2^2)+a_1b_1\af(a_2+b_2)}\\
    &\phantom{x}\x v_4^{-(a_1+b_1)(a_2\af(b_2)+b_2\af(a_2))-a_1(a_2\bt (b_2)+b_2^2\bt(a_2))-b_1(b_2\bt(a_2)+a_2^2\bt(b_2))}.
\end{align*}
\end{lemma}
\begin{proof}
Using \eqref{Eq:13}, we calculate
\begin{align*}
&x^{a_1}y^{a_2}\x x^{b_1}y^{b_2} = (x^{a_1}y^{a_2}\x x^{b_1})y^{b_2}\x (y^{b_2},x^{b_1},x^{a_1}y^{a_2})\\
&\phantom{x}=((x^{b_1+a_1}\x y^{a_2})(y^{a_2},x^{a_1},x^{b_1})\x y^{b_2})(y^{b_2},x^{b_1},x^{a_1}y^{a_2})\\
&\phantom{x}=(x^{a_1+b_1}y^{a_2}\x (y^{a_2},x^{a_1},x^{b_1})y^{b_2})(x^{a_1+b_1}y^{a_2},(y^{a_2},x^{a_1},x^{b_1}),y^{b_2})\x (y^{b_2},x^{b_1},x^{a_1}y^{a_2}).
\end{align*}
By Lemma \ref{Lm:MidA}, we can ignore the compounded associator and continue
\begin{align*}
    &[(x^{a_1+b_1}y^{a_2}\x y^{b_2})(y^{a_2},x^{a_1},x^{b_1})]
    ((y^{a_2},x^{a_1},x^{b_1}),y^{b_2},x^{a_1+b_1}y^{a_2})\x (y^{b_2},x^{b_1},x^{a_1}y^{a_2})\\
    &\phantom{x}=[(x^{a_1+b_1}y^{a_2+b_2}\x (x^{a_1+b_1},y^{a_2},y^{b_2}))(y^{a_2},x^{a_1},x^{b_1})]\\
    &\phantom{x=}\x((y^{a_2},x^{a_1},x^{b_1}),y^{b_2},x^{a_1+b_1}y^{a_2})\x (y^{b_2},x^{b_1},x^{a_1}y^{a_2}).
\end{align*}
Because associators associate with one another, we can rewrite the formula as
\begin{align*}
    x^{a_1}y^{a_2}\x x^{b_1}y^{b_2} =&
        x^{a_1+b_1}y^{a_2+b_2}\x (x^{a_1+b_1},y^{a_2},y^{b_2})(y^{a_2},x^{a_1},x^{b_1})(y^{b_2},x^{b_1},x^{a_1}y^{a_2})\\
        &\x ((y^{a_2},x^{a_1},x^{b_1}),y^{b_2},x^{a_1+b_1}y^{a_2}).
\end{align*}
Now, using Lemmas \ref{Lm:Linear} and \ref{Lm:Reduction} freely,
\begin{displaymath}
    ((y^{a_2},x^{a_1},x^{b_1}),y^{b_2},x^{a_1+b_1}y^{a_2}) = v_2^{a_1b_1a_2b_2(a_1+b_1)}v_3^{a_1b_1a_2^2b_2}.
\end{displaymath}
By Lemma \ref{Lm:AllPowers},
\begin{align*}
    (x^{a_1+b_1},y^{a_2},y^{b_2}) =& u_2^{(a_1+b_1)a_2b_2} v_2^{-\alpha(a_1+b_1)a_2b_2}v_3^{-\beta(a_1+b_1)a_2^2b_2} v_3^{-\beta(a_1+b_1)a_2b_2^2}\\
    &\x v_3^{-(a_1+b_1)\beta(a_2)b_2} v_4^{-(a_1+b_1)\alpha(a_2)b_2} v_4^{-(a_1+b_1)\beta(a_2)b_2^2}\\
    &\x v_3^{-(a_1+b_1)a_2\beta(b_2)} v_4^{-(a_1+b_1)a_2\beta(b_2)} v_4^{-(a_1+b_1)a_2\alpha(b_2)},
\end{align*}
and, similarly,
\begin{align*}
    (y^{a_2},x^{a_1},x^{b_1}) =& u_1^{-a_1b_1a_2}v_3^{\alpha(a_2)a_1b_1}v_2^{\beta(a_2)a_1^2b_1}v_2^{\beta(a_2)a_1b_1^2}\\
    &\x v_2^{a_2\beta(a_1)b_1}v_1^{a_2\alpha(a_1)b_1}v_1^{a_2\beta(a_1)b_1^2}   v_2^{a_2a_1\beta(b_1)}v_1^{a_2a_1\beta(b_1)}v_1^{a_2a_1\alpha(b_1)}.
\end{align*}
Finally, by Proposition \ref{Pr:AProduct} and \eqref{Eq:MidA}, we see that
\begin{displaymath}
    (y^{b_2},x^{b_1},x^{a_1}y^{a_2}) = (y^{b_2},x^{b_1},x^{a_1})v_2^{a_1^2b_1a_2b_2}v_2^{a_1b_1^2a_2b_2}v_3^{a_1b_1a_2b_2^2}.
\end{displaymath}
The associator $(y^{b_2},x^{b_1},x^{a_1})$ can be obtained from the already calculated associator $(y^{a_2},x^{a_1},x^{b_1})$. Putting all these associators together, we arrive at
\begin{displaymath}
    x^{a_1}y^{a_2}\x x^{b_1}y^{b_2} = x^{a_1+b_1}y^{a_2+b_2}\x u_1^{-a_1b_1(a_2+b_2)} u_2^{a_2b_2(a_1+b_1)} v_1^{c_1}v_2^{c_2}v_3^{c_3}v_4^{c_4},
\end{displaymath}
where, after summing up the exponents of the respective $v_i$s and simplifying,
\begin{align*}
    c_1 &= (a_2+b_2)(b_1\af(a_1)+a_1\af(b_1))+a_2(a_1\bt(b_1)+b_1^2\bt(a_1))\\
        &\phantom{=}+b_2(b_1\bt(a_1)+a_1^2\bt(b_1)),\\
    c_2 &= 2a_1a_2b_1b_2(a_1+b_1)+(a_2+b_2)(a_1\bt(b_1)+b_1\bt(a_1))\\
        &\phantom{=}+(\bt(a_2)+\bt(b_2))(a_1b_1^2+b_1a_1^2)-a_2b_2\af(a_1+b_1),\\
    c_3 &= -\beta(a_1+b_1)(a_2^2b_2+a_2b_2^2)-(a_1{+}b_1)(\beta(a_2)b_2+a_2\beta(b_2))\\
        &\phantom{=}+a_1b_1(\alpha(a_2)+\alpha(b_2))+a_1b_1a_2b_2(a_2+ b_2),\\
    c_4 &= -(a_1+b_1)\alpha(a_2)b_2 - (a_1+b_1)\beta(a_2)b_2^2\\
        &\phantom{=}- (a_1+b_1)a_2\beta(b_2)-(a_1+b_1)a_2\alpha(b_2).
\end{align*}
The exponents $c_1$, $c_2$ already have the desired form. To match the exponents $c_3$, $c_4$ with the formula of the lemma, note that $\alpha(a+b) = \alpha(a)+\alpha(b) + ab(a+b)$ while rewriting $c_3$, and substitute $\beta(a) = a^2-a$ into $c_4$.
\end{proof}

\begin{lemma}\label{Lm:FullMult}
In $F_3(x,y)$ we have
\begin{align*}
    &((x^{a_1}y^{a_2}\x u_1^{a_3}u_2^{a_4})\x v_1^{a_5}v_2^{a_6}v_3^{a_7}v_4^{a_8})
        \x ((x^{b_1}y^{b_2}\x u_1^{b_3}u_2^{b_4})\x v_1^{b_5}v_2^{b_6}v_3^{b_7}v_4^{b_8})\\
    &\phantom{x}= (x^{a_1+b_1}y^{a_2+b_2}\x u_1^{a_3+b_3-a_1b_1(a_2+b_2)} u_2^{a_4+b_4+a_2b_2(a_1+b_1)})\\
    &\phantom{x}\x v_1^{a_5+b_5+(a_2+b_2)(b_1\af(a_1)+a_1\af(b_1))+a_2(a_1\bt(b_1)+b_1^2\bt(a_1))+b_2(b_1\bt(a_1)+a_1^2\bt(b_1))-a_1b_1(a_3+b_3)}\\
    &\phantom{x}\x v_2^{a_6+b_6+2a_1a_2b_1b_2(a_1+b_1)+(a_2+b_2)(a_1\bt(b_1)+b_1\bt(a_1))+(\bt(a_2)+\bt(b_2))(a_1b_1^2+b_1a_1^2)}\\
    &\phantom{xxx}{}^{-a_2b_2\af(a_1+b_1)-a_1b_1(a_4+b_4)-(a_3+b_3)(a_1b_2+a_2b_1)}\\
    &\phantom{x}\x v_3^{a_7+b_7-2a_1a_2b_1b_2(a_2+b_2)-(a_1+b_1)(a_2\bt(b_2)+b_2\bt(a_2))-(\bt(a_1)+\bt(b_1))(a_2b_2^2+b_2a_2^2)}\\
    &\phantom{xxx}{}^{+a_1b_1\af(a_2+b_2)-a_2b_2(a_3+b_3)-(a_4+b_4)(a_1b_2+a_2b_1)}\\
    &\phantom{x}\x v_4^{a_8+b_8-(a_1+b_1)(a_2\af(b_2)+b_2\af(a_2))-a_1(a_2\bt(b_2)+b_2^2\bt(a_2)) -b_1(b_2\bt(a_2)+a_2^2\bt(b_2))-a_2b_2(a_4+b_4)}
\end{align*}
for every $a_i$, $b_i\in\mathbb Z$.
\end{lemma}
\begin{proof}
Using Lemma \ref{Lm:MidA} in the first step and \eqref{Eq:2A} in the second, we have
\begin{align*}
    &(x^{a_1}y^{a_2}\x u_1^{a_3}u_2^{a_4})(x^{b_1}y^{b_2}\x u_1^{b_3}u_2^{b_4})\\
    &\phantom{x}= x^{a_1}y^{a_2}\x (u_1^{a_3}u_2^{a_4}\x (x^{b_1}y^{b_2}\x u_1^{b_3}u_2^{b_4}))\\
    &\phantom{x}= x^{a_1}y^{a_2}\x (x^{b_1}y^{b_2}\x u_1^{a_3+b_3}u_2^{a_4+b_4})\\
    &\phantom{x}= (x^{a_1}y^{a_2}\x x^{b_1}y^{b_2})\x u_1^{a_3+b_3}u_2^{a_4+b_4}(x^{a_1}y^{a_2},x^{b_1}y^{b_2},u_1^{a_3+b_3}u_2^{a_4+b_4})^{-1}.
\end{align*}
Now note that  Lemma \ref{Lm:Linear} yields
\begin{displaymath}
    (x^ay^b,x^cy^d,u_1^eu_2^f) = v_1^{ace}v_2^{acf+ade+bce}v_3^{adf+bcf+bde}v_3^{bdf}.
\end{displaymath}
We are therefore done by Lemma \ref{Lm:Mult}.
\end{proof}

In the proof of the main theorem we will use a Mathematica \cite{Mathematica} code to verify certain properties of the multiplication formula of Lemma \ref{Lm:FullMult}. The code can be downloaded from the website of the third-named author.

\begin{theorem}\label{Th:Main}
Let $F_3(x,y)$ be the free commutative automorphic loop of nilpotency class $3$ on free generators $x$, $y$. Let $u_1 = (x,x,y)$, $u_2 = (x,y,y)$, $v_1 = (x,x,u_1)$, $v_2 = (x,x,u_2)$, $v_3 = (y,y,u_1)$, $v_4=(y,y,u_2)$. Then each element of $F_3(x,y)$ can be written uniquely as $(x^{a_1}y^{a_2}\x u_1^{a_3}u_2^{a_4})v_1^{a_5}v_2^{a_6}v_3^{a_7}v_4^{a_8}$, and $F_3(x,y)$ is isomorphic to $(\mathbb Z^8,*)$, where the multiplication $*$ of exponents is as in Lemma \ref{Lm:FullMult}.
\end{theorem}
\begin{proof}
Let $F$ be defined on $\mathbb Z^8$ with multiplication according to Lemma \ref{Lm:FullMult}. Denote by $e_i$ the element of $\mathbb Z^8$ whose only non-zero coordinate is equal to $1$ and is located in position $i$. Straightforward calculation in Mathematica shows that $F$ is a loop with identity element $(0,0,0,0,0,0,0,0)$ such that $(e_1,e_1,e_2) = e_3$, $(e_1,e_2,e_2) = e_4$, $(e_1,e_1,e_3) = e_5$, $(e_1,e_1,e_4) = e_6$, $(e_2,e_2,e_3) = e_7$ and $(e_2,e_2,e_4) = e_8$. Moreover, $F$ is a commutative automorphic loop. (To verify that $F$ is automorphic, the code merely needs to check by symbolic calculation that the inner mappings $L_{a,b}$ are automorphisms of $F$.)

We claim that $F_3(x,y)$ is isomorphic to $F$. Let $f:F_3(x,y)\to F$ be the homomorphism determined by $f(x) = e_1$, $f(y)=e_2$. Because homomorphisms behave well on associators, namely $f((a,b,c)) = (f(a),f(b),f(c))$, the calculation in the previous paragraph shows that $f(u_1)=e_3$, $f(u_2)=e_4$, $f(v_1)=e_5$, $f(v_2)=e_6$, $f(v_3)=e_7$ and $f(v_4)=e_8$. By Lemma \ref{Lm:Canonical}, any element $w$ of $F_3(x,y)$ can be written as $w=(x^{a_1}y^{a_2}\x u_1^{a_3}u_2^{a_4})v_1^{a_5}v_2^{a_6}v_3^{a_7}v_4^{a_8}$, and it now follows that $f(w) = (a_1,a_2,a_3,a_4,a_5,a_6,a_7,a_8)$. This means that $f$ is onto $F$, and also that the exponents $a_i$ in the decomposition of $w$ are uniquely determined by $w$. Hence $f:F_3(x,y)\to F$ is an isomorphism.
\end{proof}

We conclude the paper with some structural information about $F_3(x,y)$.

\begin{proposition}
Let $Q = F_3(x,y)$ be identified with $(\mathbb Z^8,*)$ as in Theorem \ref{Th:Main}. Then $A(Q) = N_\mu(Q) = 0\times 0\times \mathbb Z^6$ and $N_\lambda(Q) = N_\rho(Q) = N(Q) = Z(Q) = 0\times 0\times 0\times 0\times \mathbb Z^4$.
\end{proposition}
\begin{proof}
We already know from Lemma \ref{Lm:MidA} that $A(Q)\le N_\mu(Q)$. By Proposition \ref{Pr:AProduct} and Lemmas \ref{Lm:AllPowers}, \ref{Lm:Reduction},
\begin{align*}
    &(x,x^{a_1}y^{a_2},y) =(x,x^{a_1},y)(x,y^{a_2},y)\\
    &\phantom{x=}\x ((x,x^{a_1},y),x^{a_1},y^{a_2})((x,y^{a_2},y),y^{a_2},x^{a_1})((x,x^{a_1},y),y^{a_2},x)\\
    &\phantom{x=}\x ((x,y^{a_2},y),x^{a_1},x)((x,x^{a_1},y),y^{a_2},y)((x,y^{a_2},y),x^{a_1},y)\\
    &\phantom{x}=(x,x^{a_1},y)(x,y^{a_2},y)v_2^{-a_1^2a_2-2a_1a_2}v_3^{-a_1a_2^2-2a_1a_2}\\
    &\phantom{x}=u_1^{a_1}(u_1,x,x)^{\af(a_1)+\bt(a_1)}(u_1,x,y)^{\bt(a_1)}\x u_2^{a_2}(u_2,y,y)^{\af(a_2)+\bt(a_2)}(u_2,y,x)^{\bt(a_2)}\\
    &\phantom{x=}\x v_2^{-a_1^2a_2-2a_1a_2}v_3^{-a_1a_2^2-2a_1a_2}\\
    &\phantom{x}=u_1^{a_1}u_2^{a_2}v_1^{-\af(a_1)-\bt(a_1)}v_2^{-\bt(a_1)-a_1^2a_2-2a_1a_2}
        v_3^{-\bt(a_2)-a_1a_2^2-2a_1a_2}v_4^{-\af(a_2)-\bt(a_2)}.
\end{align*}
Thus, if either $a_1\neq 0$ or $a_2\neq 0$ then $r=(x,x^{a_1}y^{a_2},y)\neq 1$. In other words, if $r\in N_\mu(Q)$ then $r\in 0\times 0 \times\z^6$. We conclude, $A(Q)\leq N_\mu(Q)\leq 0\times 0\times\z^6\leq A(Q)$, so $N_\mu(Q)=A(Q)=0\times 0 \times \z^6$.

Since $Q$ has nilpotency class $3$, we have $0\times 0\times 0\times 0\times \z^4\leq\Z(Q)\leq N_\rho(Q)= N_\lambda(Q)$. Now,
\begin{align*}
    &(x,x,x^{a_1}y^{a_2}\x u_1^{a_3}u_2^{a_4})=(x,x,x^{a_1}y^{a_2})(x,x,u_1^{a_3}u_2^{a_4})=(x,x,x^{a_1}y^{a_2})v_1^{a_3}v_2^{a_4}\\
    &\phantom{x}=[(x,x,y^{a_2})((x,x,y^{a_2}),y^{a_2},x^{a_1})((x,x,y^{a_2}),x^{a_1},x)((x,x,y^{a_2}),x^{a_1},x)]v_1^{a_3}v_2^{a_4}\\
    &\phantom{x}=(x,x,y^{a_2})v_1^{a_3-2a_1a_2}v_2^{a_4-a_1a_2^2}\\
    &\phantom{x}=u_2^{a_2}\x (u_2,y,y)^{\af(a_2)}(u_2,y,x)^{2\bt(a_2)}v_1^{a_3-2a_1a_2}v_2^{a_4-a_1a_2^2}\\
    &\phantom{x}=u_2^{a_2}v_1^{a_3-2a_1a_2}v_2^{a_4-a_1a_2^2}v_3^{2\bt(a_2)}v_4^{\af(a_2)}.
\end{align*}
If $(x,x,x^{a_1}y^{a_2}\x u_1^{a_3}u_2^{a_4})=1$, $a_2$ must be zero. Then $v_1^{a_3}v_2^{a_4}=1$ and thus $a_3=a_4=0$. Therefore, if $(x,x,x^{a_1}y^{a_2}\x w^{a_3}t^{a_4})=1$ then $a_2=a_3=a_4=0$. Finally,
\begin{displaymath}
    (y,x,x^{a_1})=u_1^{-a_1}(u_1^{-1},x,x)^{\af(a_1)+\bt(a_1)}(u_1^{-1},x,y)^{\bt(a_1)}=u_1^{-a_1}v_1^{\af(a_1)+\bt(a_1)}v_2^{\bt(a_1)}.
\end{displaymath}
So, $(y,x,x^{a_1})=1$ implies $a_1=0$. Summarizing, $x^{a_1}y^{a_2}\x u_1^{a_3}u_2^{a_4}\in N_\rho(Q)$ if and only if $a_1=a_2=a_3=a_4=0$. Hence $N_\rho(Q)= \Z(Q)=0\times 0\times 0\times 0\times\z^4$.
\end{proof}

\end{document}